%% file: main_arxiv.tex
\documentclass[a4paper,twoside,11pt]{article}
\usepackage[utf8]{inputenc}
\usepackage{graphicx,epstopdf}
\usepackage{amsmath}
\usepackage{amsthm}
\usepackage{amssymb}
\usepackage{color}
\usepackage{mathtools}
\usepackage{subfig}
\usepackage{siunitx}
\usepackage{amsthm}
\usepackage[font=footnotesize,labelfont=bf]{caption}

\usepackage{geometry,calc}
\usepackage{listings}

\input{defs}
\title{Stochastic Production Planning in Manufacturing Systems}
\author{Dragos-Patru Covei\thanks{Department of Applied Mathematics, The Bucharest University of Economic Studies, Piata Romana, No. 6, Bucharest, District 1, 010374, Romania}}

\date{}

\begin{document}
\maketitle
	
\begin{abstract}
\noindent
\noindent We extend the stochastic production planning framework to
manufacturing systems, where the set of admissible production configurations
is described by a general smooth convex domain $\omega $. In our setting,
production operations continue as long as the production inventory $y(t)$
remains inside the capacity limits of $\omega $ and are halted once the
state exits this region, i.e.,%
\begin{equation*}
\tau =\inf \{t>0:\Vert y(t)-x_{0}\Vert >\text{dist}(x_{0},\partial \omega
)\}.
\end{equation*}%
The running cost is partitioned into a quadratic production cost $%
a(p)=\left\Vert p\right\Vert ^{2}$ and an inventory holding cost modeled by
a positive continuous function $b(y)$. We derive the associated
Hamilton--Jacobi--Bellman (HJB) equation, verify the supermartingale
property of the value function, and characterize the optimal feedback
control. Techniques inspired by Lasry, Lions and Alvarez enable us to prove
existence and uniqueness within this generalized production planning
framework. Numerical experiments and a real-world examples illustrate the
practical relevance of our results.

		
\end{abstract}

\section{Introduction}

Consider a convex $C^{2}$, open, bounded domain $\omega \subset \mathbb{R}%
^{N}$ that represents the set of admissible production configurations in a
manufacturing system. Its smooth boundary $\partial \omega $ serves as a
threshold that together with a chosen interior reference point $x_{0}\in
\omega $ defines the operational capacity limits. In this paper, we study a
stochastic production planning problem in which the evolution of the
production inventory is subject to these capacity constraints. Production
halts when the inventory state exits the region $\omega $; that is, when it
reaches the threshold $\text{dist}(x_{0},\partial \omega )$, thereby
coupling the state dynamics with the geometry of the production system.

This approach not only extends the classical production planning models in 
\cite{Cadenillas2010, Cadenillas2013, Gharbi2003, Sethi1981, Thompson1984}
but also builds on methodologies that have proven effective in optimal
control and related applications such as image restoration \cite%
{CoveiImageRestoration} and in earlier studies on stochastic production
planning \cite{CCP2, CoveiProductionPlanning}.

Building upon these foundational ideas, our investigation is motivated by
the increasing complexity and uncertainty in manufacturing systems. Recent
advances in stochastic optimization and control theory have demonstrated the
benefits of integrating geometrical constraints directly into the modeling
of production systems \cite{Bayer2006, Bayram1987, Holmes2022}. In
particular, the use of Hamilton--Jacobi--Bellman equations in conjunction
with convex analysis (as pioneered by Brascamp and Lieb \cite{BL} and
further refined by Korevaar \cite{Korevaar}) has provided deep insights into
the structure of optimal policies. Our work employs these techniques to
derive rigorous existence and uniqueness results for the associated value
function, ensuring that the optimal control policy remains robust under
random fluctuations.

Furthermore, the analysis presented here extends earlier approaches of \cite%
{Alvarez1996, Lasry1989} by incorporating advanced sub- and supersolution
methods \cite{Amann1976} alongside classical elliptic theory \cite{GT1983}.
This hybrid methodology allows us to establish important qualitative
properties of the solution, including convexity, and the concavity
characteristics of the transformed value function. Such properties are
instrumental in not only ensuring stability but also in facilitating
efficient numerical approximations, as evidenced in recent computational
studies \cite{Holmes2022}.

Our results also resonate with recent innovations in production planning
models, where the interplay between the geometry of the operational domain $%
\omega $ and the stochastic dynamics of the inventory process yields
nontrivial insights into optimal control strategies. In particular, we prove
that, under suitable conditions, the value function exhibits a concavity
property over $\omega $, which in turn guarantees the smoothness and
robustness of the derived optimal feedback control law. This contribution
not only strengthens the theoretical underpinnings of stochastic production
planning but also aligns with practical observations in manufacturing
environments.

The remainder of this paper is organized as follows. In Section \ref{2}, we
detail the precise mathematical formulation of the stochastic production
planning problem, including the inventory dynamics and cost functional.
Section \ref{3} presents our main analytical results, where we prove the
existence, uniqueness, and qualitative properties of the value function. In
Section \ref{4}, we derive the optimal feedback control law through a
rigorous verification argument involving the martingale property. Finally,
Section \ref{5} illustrates our approach with numerical experiments, and
Section \ref{6} concludes with potential extensions and open research
directions. The paper document concludes with an Appendix Section \ref{ap}
that provides the Python codes implementing our results.

\subsection{Generalized Hamilton--Jacobi--Bellman Equation\label{2}}

The problem formulation involves:

\begin{itemize}
\item \textbf{Production configurations:} Here, $\omega \subset \mathbb{R}%
^{N}$ is a convex, open, and bounded domain that represents the set of
admissible production configurations, and $x_{0}\in \omega $ is a chosen
interior reference point. The boundary $\partial \omega $ defines the
operational capacity limits of our manufacturing system.

\item \textbf{Inventory Dynamics:} The production inventory levels $%
y(t)=(y_{1}(t),\dots ,y_{N}(t))$ evolve according to 
\begin{equation*}
dy_{i}(t)=p_{i}(t)\,dt+\sigma \,dw_{i}(t),\quad y_{i}(0)=y_{i}^{0},\quad
i=1,\ldots ,N,
\end{equation*}%
where $p(t)=(p_{1}(t),\dots ,p_{N}(t))\in \mathbb{R}^{N}$ is the vector of
production rates and $\sigma >0$ is the volatility parameter representing
random fluctuations in the manufacturing process.

\item \textbf{Cost Functional:} The objective is to minimize the total
expected cost given by 
\begin{equation*}
J(p)=\mathbb{E}\int_{0}^{\tau }\Bigl[|p(t)|^{2}+b(y(t))\Bigr]dt,
\end{equation*}%
where $b(y)$ is a positive continuous function capturing the inventory
holding cost.

\item \textbf{Stopping Condition:} Production is halted at the stopping time 
\begin{equation}
\tau =\inf \{t>0:\Vert y(t)-x_{0}\Vert >\text{dist}(x_{0},\partial \omega
)\},  \label{st}
\end{equation}%
where $\Vert \circ \Vert $ denotes the Euclidean norm, and $\text{dist}%
(x_{0},\partial \omega )$ is the Euclidean distance between the reference
point $x_{0}$ and the boundary $\partial \omega $.
\end{itemize}

The stochastic process $w(t)$ is an $N$-dimensional Brownian motion defined
on a complete probability space $(\Omega ,\mathcal{F},\{\mathcal{F}%
_{t}\}_{t\geq 0},P)$.

Our primary goal is to derive and analyze the Hamilton--Jacobi--Bellman
(HJB) equation corresponding to this production planning problem, to
characterize the optimal control policy, and to verify key structural
properties of the resulting inventory process. Let the value function $z(y)$
represent the minimum total production cost achievable starting at the
inventory state $y$; that is,%
\begin{equation*}
z(y)=\inf_{p}J(p),
\end{equation*}%
where $J(p)$ denotes the total expected cost comprising the quadratic
production cost and the inventory holding cost $b(y)$.

By the dynamic programming principle, $z(y)$ satisfies the
Hamilton--Jacobi--Bellman (HJB) equation:%
\begin{equation}
-\frac{\sigma ^{2}}{2}\Delta z(y)-b(y)=\inf_{p}\left\{ p\cdot \nabla
z(y)+|p|^{2}\right\} ,\quad y\in \omega .  \label{hjb}
\end{equation}%
Here, $b(y)$ is a continuous function that models the inventory cost in the
production process, while $\sigma >0$ captures volatility in the inventory
dynamics.

Minimizing the right-hand side over the production rate $p$ yields the
optimal control policy%
\begin{equation*}
p^{\ast }(y)=-\frac{1}{2}\nabla z(y).
\end{equation*}%
Substituting $p^{\ast }(y)$ into the HJB equation leads to

\begin{equation*}
-\frac{\sigma ^{2}}{2}\Delta z(y)-b(y)=-\frac{1}{4}|\nabla z(y)|^{2},\quad
y\in \omega .
\end{equation*}%
The boundary condition is imposed as

\begin{equation*}
z(y)=Z_{0},\quad \text{for }y\in \partial \omega ,
\end{equation*}%
where $Z_{0}\geq 0$ is a constant determined by economic or operational
constraints.

Defining the transformations

\begin{equation*}
z(y)=-v(y),\quad u(y)=e^{\frac{v(y)}{2\sigma ^{2}}},
\end{equation*}%
and applying the chain rule, we obtain

\begin{equation*}
\nabla u(y)=\frac{u(y)}{2\sigma ^{2}}\nabla v(y),\quad \Delta u(y)=\frac{u(y)%
}{2\sigma ^{2}}\Delta v(y)+\frac{u(y)}{(2\sigma ^{2})^{2}}|\nabla v(y)|^{2}.
\end{equation*}%
Hence, the HJB equation (\ref{hjb}) is transformed into the linear partial
differential equation 
\begin{equation}
\Delta u(y)=\frac{1}{\sigma ^{4}}b(y)u(y),\quad y\in \omega ,  \label{BVP0}
\end{equation}%
subject to the boundary condition 
\begin{equation}
u(y)=e^{-\frac{Z_{0}}{2\sigma ^{2}}},\quad y\in \partial \omega .
\label{BVP1}
\end{equation}

\section{Results on the Value Function \label{3}}

We begin by recalling some basic definitions that are fundamental to our
existence results of the value function.

\subsection{Preliminaries}

In what follows, let $\omega \subset \mathbb{R}^{N}$ be a $C^{2}$, open
bounded domain with smooth boundary $\partial \omega $ representing the set
of admissible production inventory states. Assume that the inventory holding
cost function $b:\overline{\omega }\rightarrow \lbrack 0,\infty )$ is
continuous with its zero set exactly the closure of an interior connected
subdomain $\omega _{0}$ $\Subset \omega $ (i.e. $b\left( y\right) =0$ if and
only if $y\in $ $\omega _{0}$ and $b\left( y\right) >0$ for all $y\in $ $%
\overline{\omega }\backslash \omega _{0}$).

\begin{definition}
A function $\overline{u}\in C^{2}(\omega )\cap C^{1}(\overline{\omega })$ is
called a \emph{supersolution} of \eqref{BVP0}-\eqref{BVP1} if%
\begin{equation}
\Delta \overline{u}(y)\leq \frac{1}{\sigma ^{4}}\,b(y)\,\overline{u}(y)\quad 
\text{in }\omega ,\quad \text{and}\quad \overline{u}(y)=e^{-Z_{0}/(2\sigma
^{2})}\quad \text{on }\partial \omega ,  \label{super}
\end{equation}%
where the boundary condition reflects the cost level at which production is
halted. Similarly, a function $\underline{u}\in C^{2}(\omega )\cap C^{1}(%
\overline{\omega })$ is called a \emph{subsolution} of \eqref{BVP0}-%
\eqref{BVP1} if%
\begin{equation}
\Delta \underline{u}(y)\geq \frac{1}{\sigma ^{4}}\,b(y)\,\underline{u}%
(y)\quad \text{in }\omega ,\quad \text{and}\quad \underline{u}%
(y)=e^{-Z_{0}/(2\sigma ^{2})}\quad \text{on }\partial \omega .  \label{sub}
\end{equation}
\end{definition}

Having established that there exist a subsolution $\underline{u}$ and a
supersolution $\overline{u}$ satisfying

\begin{equation*}
\underline{u}(y)\leq \overline{u}(y)\quad \text{for all }y\in \overline{%
\omega },
\end{equation*}%
the standard theory of monotone iteration (see, e.g., \cite{Amann1976})
applies. An iterative sequence $\{u_{n}\}$ can then be constructed so that

\begin{equation*}
\underline{u}(y)\leq \ldots \leq u_{n}(y)\leq u_{n+1}(y)\leq \ldots \leq 
\overline{u}(y),
\end{equation*}%
and in the limit the sequence converges to a positive solution $u(y)\in
C^{2}(\omega )\cap C^{1}(\overline{\omega })$ of \eqref{BVP0}-\eqref{BVP1}
satisfying

\begin{equation*}
\underline{u}(y)\leq u(y)\leq \overline{u}(y)\quad \text{for all }y\in 
\overline{\omega }.
\end{equation*}%
This completes the proof of the existence of a positive solution via the
method of sub-- and supersolutions in our production planning context. In
the next subsection, we present and prove our findings.

\subsection{Existence and Uniqueness}

The first result addresses the existence and qualitative properties of the
linear partial differential equation obtained by transforming the equation
associated with the value function.

\begin{theorem}
\label{exist} Suppose $\omega \subset \mathbb{R}^{N}$ is a convex $C^{2}$,
open bounded domain representing the admissible set of production inventory
states, with smooth boundary $\partial \omega $, and that $b:\overline{%
\omega }\rightarrow \lbrack 0,\infty )$ is continuous with its zero set
exactly the closure of an interior connected subdomain $\omega _{0}\subset
\omega $. Then, there exists a unique, positive, convex value function $u\in
C^{2}(\omega )\cap C^{1}(\overline{\omega })$ satisfying 
\begin{equation}
\left\{ 
\begin{array}{lll}
\Delta u(y)=\frac{1}{\sigma ^{4}}b(y)u(y), & for & y\in \omega , \\ 
u(y)=e^{-\frac{Z_{0}}{2\sigma ^{2}}}, & for & y\in \partial \omega .%
\end{array}%
\right.  \label{BVP2}
\end{equation}%
Moreover, the asymptotic behavior of the magnitude of the optimal control is
characterized by%
\begin{equation}
\lim_{y\rightarrow \partial \omega }\Vert -\frac{\nabla z(y)}{2}\Vert
=\sigma ^{2}\,\frac{\gamma }{\exp \left( -\frac{Z_{0}}{2\sigma ^{2}}\right) }%
=\sigma ^{2}\,\gamma \,\exp \!\left( \frac{Z_{0}}{2\sigma ^{2}}\right) ,
\label{as}
\end{equation}%
where we define $\gamma :=-\partial _{n}u(y_{0})\in \left( 0,\infty \right) $%
, with $\partial _{n}u(y_{0})$ denoting the outward normal derivative of $u$
at $y_{0}\in \partial \omega $.
\end{theorem}

\begin{proof}
The existence of the solution $u$ is established via the sub-- and
supersolution method described above. Define the constant function%
\begin{equation*}
\overline{u}(y)=e^{-Z_{0}/(2\sigma ^{2})}\quad \text{for all }y\in \overline{%
\omega }.
\end{equation*}%
Since $\overline{u}$ is constant, we immediately have $\Delta \overline{u}%
(y)=0$ in $\omega $, so that%
\begin{equation*}
\Delta \overline{u}(y)-\frac{1}{\sigma ^{4}}\,b(y)\,\overline{u}(y)=-\frac{1%
}{\sigma ^{4}}\,b(y)\,e^{-Z_{0}/(2\sigma ^{2})}\leq 0.
\end{equation*}%
Moreover, $\overline{u}$ clearly satisfies the boundary condition, hence it
is a supersolution of \eqref{BVP2}. Next, by classical elliptic theory (see,
e.g., \cite[Theorem 4.3]{GT1983}), the auxiliary problem%
\begin{equation}
\left\{ 
\begin{array}{lll}
\Delta w(y)=\frac{1}{\sigma ^{4}}b(y), & for & y\in \omega , \\ 
w(y)=e^{-\frac{Z_{0}}{2\sigma ^{2}}}, & for & y\in \partial \omega ,%
\end{array}%
\right.  \label{AUX}
\end{equation}%
has a unique positive solution $w(y)\in C^{2}\left( \omega \right) \cap
C^{1}\left( \overline{\omega }\right) $. Using the maximum principle applied
to \eqref{AUX}, one checks that $w(y)$ satisfies both the boundary condition
and the differential inequality (\ref{sub}). In particular, one may show that%
\begin{equation*}
w(y)\leq e^{-\frac{Z_{0}}{2\sigma ^{2}}}=\overline{u}(y)\quad \text{in }%
\omega .
\end{equation*}%
We now define $\underline{u}(y)=\,w(y)$. For $\underline{u}(y)$ to serve as
a subsolution, we require two properties:

\textbf{(i) Boundary condition:}\ 

Since%
\begin{equation*}
w(y)=e^{-Z_{0}/(2\sigma ^{2})}\quad \text{on }\partial \omega ,\text{ we
have }\underline{u}(y)=\,e^{-Z_{0}/(2\sigma ^{2})}\quad \text{on }\partial
\omega .
\end{equation*}

\textbf{(ii) Differential inequality in $\omega$:}\ 

In the interior of $\omega $, we compute%
\begin{equation*}
\Delta \underline{u}(y)=\Delta w(y)=\frac{1}{\sigma ^{4}}b(y),
\end{equation*}%
since $w$ satisfies $\Delta w(y)=\frac{1}{\sigma ^{4}}b(y)$. Hence,%
\begin{equation*}
\Delta \underline{u}(y)-\frac{1}{\sigma ^{4}}b(y)\underline{u}(y)=\frac{1}{%
\sigma ^{4}}b(y)-\frac{1}{\sigma ^{4}}b(y)(\,w(y))=\frac{1}{\sigma ^{4}}b(y)%
\left[ 1-w(y)\right] .
\end{equation*}%
By the maximum principle applied to the linear problem \eqref{AUX}, one
deduces that%
\begin{equation*}
\underline{u}(y)=w(y)\leq e^{-Z_{0}/(2\sigma ^{2})}=\overline{u}(y)\quad 
\text{in }\omega .
\end{equation*}%
Assuming, for instance, that $e^{-Z_{0}/(2\sigma ^{2})}\leq 1$ (which holds
when $Z_{0}\geq 0$), it follows that%
\begin{equation*}
w(y)\leq 1\quad \text{for all }y\in \omega .
\end{equation*}%
Thus,%
\begin{equation*}
1-w(y)\geq 0\quad \text{in }\omega ,
\end{equation*}%
and therefore,%
\begin{equation*}
\Delta \underline{u}(y)-\frac{1}{\sigma ^{4}}\,b(y)\,\underline{u}(y)=\frac{1%
}{\sigma ^{4}}\,b(y)\Bigl[1-w(y)\Bigr]\geq 0.
\end{equation*}%
This shows that $\underline{u}(y)$ is a subsolution of \eqref{BVP2}. By
construction, we have%
\begin{equation*}
\underline{u}(y)\leq u(y)\leq \overline{u}(y)\quad \text{in }\overline{%
\omega },
\end{equation*}%
ensuring that the solution $u\in C^{2}(\omega )\cap C^{1}(\overline{\omega }%
) $ exists within these bounds. Monotone iteration techniques then yield the
desired positive solution $u(y)$ of \eqref{BVP2} satisfying%
\begin{equation*}
\underline{u}(y)\leq u(y)\leq \overline{u}(y)\quad \text{for all }y\in 
\overline{\omega }.
\end{equation*}%
Therefore, by the \emph{strong maximum principle} (and its counterpart for
the minimum), the solution $u(y)$ is unique. This approach was introduced to
both implement the results in Python and ensure the sequence remains bounded
during the convexity phase---although the convexity proof will be
established in the next step. Let $B_{B_{R_{d}}}\subset \mathbb{R}^{N}$ be a
ball of radius $B_{R_{d}}>0$. Define the functions%
\begin{equation*}
\overline{b}(y)=\max_{\Vert t\Vert =y}b(t)\quad \text{and}\quad \underline{b}%
(y)=\min_{\Vert t\Vert =y}b(t).
\end{equation*}%
Results in the literature (see, e.g., \cite{CCP2, CoveiProductionPlanning})
guarantee that the solutions $\underline{z}$ and $\overline{z}$ of the
problems 
\begin{equation*}
\left\{ 
\begin{array}{lll}
\Delta \underline{z}(y)=\frac{1}{\sigma ^{4}}\overline{b}(y)\underline{z}(y),
& for & y\in B_{R_{d}}, \\ 
\underline{z}(y)=e^{-\frac{Z_{0}}{2\sigma ^{2}}}, & for & y\in \partial
B_{R_{d}},%
\end{array}%
\right. \text{ and }\left\{ 
\begin{array}{lll}
\Delta \overline{z}(y)=\frac{1}{\sigma ^{4}}\underline{b}(y)\overline{z}(y),
& for & y\in B_{R_{d}}, \\ 
\overline{z}(y)=e^{-\frac{Z_{0}}{2\sigma ^{2}}}, & for & y\in \partial
B_{R_{d}},%
\end{array}%
\right.
\end{equation*}%
are convex in $B_{B_{R_{d}}}$. Moreover, since the smooth domain $\omega $
can be covered by suitable balls, i.e.,%
\begin{equation*}
\omega =\bigcup_{d=0}^{\infty }B_{R_{d}},
\end{equation*}%
these local convexity properties carry over to $\omega $ in a patchwork
argument. Inspired by the sub- and supersolution method above, we now
construct an iterative sequence $\{z_{d}\}_{d\geq 0}$ on $\omega $. To be
precise, fix an exhaustion $\omega =\bigcup_{d=0}^{\infty }B_{R_{d}}$ and
define the initial function by%
\begin{equation*}
z_{0}(y)=\underline{z}(y)\text{ for }y\in B_{R_{0}};\quad \text{in
particular, }z_{0}(y)=e^{-\frac{Z_{0}}{2\sigma ^{2}}}\text{ on }\partial
B_{R_{0}}.
\end{equation*}%
Then, for $d\geq 0$, define $z_{d+1}$ as the solution of%
\begin{equation}
\left\{ 
\begin{array}{lll}
\Delta z_{d+1}(y)=\dfrac{1}{\sigma ^{4}}b(y)\,z_{d}(y), & in & B_{R_{d+1}},
\\ 
z_{d+1}(y)=e^{-\frac{Z_{0}}{2\sigma ^{2}}}, & on & \partial B_{R_{d+1}},%
\end{array}%
\right.  \label{iteration}
\end{equation}%
noting that the prescribed constant boundary value guarantees that%
\begin{equation*}
z_{d+1}(y)\big|_{\partial B_{R_{d+1}}}=z_{d}(y)\big|_{\partial
B_{R_{d}}}\quad \text{for every }d\geq 0.
\end{equation*}%
Since $z_{0}$ is convex and the forcing term in the above problem is
nonnegative, standard arguments based on the \emph{convexity of the domain
and maximum principle} imply that if $z_{d}$ is convex then so is $z_{d+1}$.
By induction, every iterate $z_{d}$ is convex on the convex domain $\omega $%
. Standard monotone iteration methods then imply that the sequence $%
\{z_{d}\} $ converges (say, in the $C^{2}$-topology) to a function $%
z(y)=\lim_{d\rightarrow \infty }z_{d}(y)$, which satisfies%
\begin{equation*}
\Delta z(y)=\frac{1}{\sigma ^{4}}\,b(y)\,z(y),\quad y\in \omega ,\qquad
z(y)=e^{-\frac{Z_{0}}{2\sigma ^{2}}}\quad \text{for }y\in \partial \omega .
\end{equation*}%
Since convexity is preserved under uniform convergence of the Hessians, the
limiting function $z$ is convex. By the uniqueness of the solution to %
\eqref{BVP2} (the main boundary value problem), we deduce that $z=u$, which
shows that the solution $u$ is convex on $\omega $.

To develop the asymptotic behavior (\ref{as}), we utilize the techniques
introduced in \cite{Lieberman2015, Lieberman2020, Lieberman2011}. We define
the value function by%
\begin{equation*}
z(y)=-2\sigma ^{2}\ln u(y),
\end{equation*}%
so that by the chain rule,%
\begin{equation*}
\nabla z(y)=-2\sigma ^{2}\,\frac{\nabla u(y)}{u(y)},\text{ and consequently, 
}\Vert \nabla z(y)\Vert =2\sigma ^{2}\,\frac{\Vert \nabla u(y)\Vert }{u(y)}.
\end{equation*}%
We know that the function $u$ is twice continuously differentiable in a
neighborhood of the boundary $\partial \omega $ and satisfies the Dirichlet
condition%
\begin{equation*}
u(y)=\exp \left( -\frac{Z_{0}}{2\sigma ^{2}}\right) ,\quad x\in \partial
\omega ,\quad Z_{0}\geq 0.
\end{equation*}%
Let $y_{0}\in \partial \omega $ be an arbitrary boundary point and denote by%
\begin{equation*}
d\left( y\right) =\text{dist}\left( y,\partial \omega \right)
\end{equation*}%
the distance from $y$ to the boundary $\partial \omega $. Since $u$ is
smooth up to the boundary, for $y$ near $y_{0}$ a first-order Taylor
expansion in the outward normal direction yields%
\begin{equation*}
u(y)=u(y_{0})+\partial _{n}u(y_{0})\,d(y)+o(d(y)),
\end{equation*}%
where $u(y_{0})=\exp \left( -\frac{Z_{0}}{2\sigma ^{2}}\right) $ and $%
\partial _{n}u(y_{0})$ is the outward normal derivative at $y_{0}$. By the
Hopf lemma, we have $\partial _{n}u(y_{0})<0$. Define%
\begin{equation*}
\gamma :=-\partial _{n}u(y_{0})>0.
\end{equation*}%
Thus, we can write%
\begin{equation*}
u(y)=\exp \left( -\frac{Z_{0}}{2\sigma ^{2}}\right) +\gamma \,d(y)+o(d(y)).
\end{equation*}%
In addition, it follows that%
\begin{equation*}
\Vert \nabla u(y)\Vert =\gamma +o(1)\quad \text{as }d(y)\rightarrow 0.
\end{equation*}%
Substituting these expressions into the gradient estimate for $z(x)$ gives%
\begin{equation*}
\Vert \nabla z(y)\Vert =2\sigma ^{2}\,\frac{\Vert \nabla u(y)\Vert }{u(y)}%
=2\sigma ^{2}\,\frac{\gamma +o(1)}{\exp \left( -\frac{Z_{0}}{2\sigma ^{2}}%
\right) +\gamma \,d(y)+o(d(y))}.
\end{equation*}%
Taking the limit as $d(y)\rightarrow 0$ (i.e. as $y\rightarrow \partial
\omega $), we obtain%
\begin{equation*}
\lim_{y\rightarrow \partial \omega }\Vert \nabla z(y)\Vert =2\sigma ^{2}\,%
\frac{\gamma }{\exp \left( -\frac{Z_{0}}{2\sigma ^{2}}\right) }=2\sigma
^{2}\,\gamma \,\exp \!\left( \frac{Z_{0}}{2\sigma ^{2}}\right) .
\end{equation*}%
Thus, the asymptotic estimate is%
\begin{equation*}
\frac{\Vert \nabla u(y)\Vert }{u\left( y\right) }\sim 2\sigma ^{2}\,\gamma
\quad \text{as }y\rightarrow \partial \omega ,
\end{equation*}%
which correspond to the radial case obtained in \cite{CCP2,
CoveiProductionPlanning}.
\end{proof}

Our second result establishes conditions under which the value function
\textquotedblright in our production planning framework\textquotedblright\
exhibits concavity. The analysis below is inspired by convexity and
gradient-estimate techniques as developed in \cite%
{Lieberman2015,Lieberman2020,Lieberman2011}. While this result appears in
the literature, incorporating it into our paper is essential for practical
considerations. Consequently, we have adapted the proof to suit the context
of our work.

\begin{theorem}
\label{thm:logconvex} Let $\omega \subset \mathbb{R}^{N}$ be a $C^{2}$, open
bounded convex domain representing the production inventory space, with
smooth boundary $\partial \omega $. Let $u:\omega \rightarrow (0,\infty )$
be a twice continuously differentiable function that is convex on $\omega $.
Suppose that for every $y\in \omega $ and every $\xi \in \mathbb{R}^{N}$ the
one-dimensional slice%
\begin{equation*}
\varphi (t)=u(y+t\,\xi ),\quad t\in I,\quad \text{with }I=\{t\in \mathbb{R}%
:y+t\,\xi \in \omega \},
\end{equation*}%
satisfies the inequality%
\begin{equation*}
\varphi (t)\,\varphi ^{\prime \prime }(t)\geq \bigl[\varphi ^{\prime }(t)%
\bigr]^{2}\quad \text{for all }t\in I.
\end{equation*}%
Then $\log u$ is convex on $\omega $, and hence the value function defined by%
\begin{equation*}
z(y)=-2\sigma ^{2}\log u(y)
\end{equation*}%
is concave on $\omega $.
\end{theorem}

\begin{proof}
Fix an arbitrary point $y\in \omega $ and an arbitrary direction $\xi \in 
\mathbb{R}^{N}$. Define%
\begin{equation*}
\varphi (t)=u(y+t\,\xi ),\quad t\in I,
\end{equation*}%
with $I=\{t\in \mathbb{R}:y+t\,\xi \in \omega \}$. Since $u>0$ and is twice
differentiable, $\varphi $ is positive and twice differentiable on $I$.

Define also%
\begin{equation*}
\psi (t)=\log \varphi (t)=\log u(y+t\,\xi ).
\end{equation*}%
Differentiating, we obtain:%
\begin{equation*}
\psi ^{\prime }(t)=\frac{\varphi ^{\prime }(t)}{\varphi (t)},
\end{equation*}%
and upon differentiating once more,%
\begin{equation*}
\psi ^{\prime \prime }(t)=\frac{\varphi ^{\prime \prime }(t)\,\varphi (t)-%
\bigl[\varphi ^{\prime }(t)\bigr]^{2}}{\varphi (t)^{2}}.
\end{equation*}%
By hypothesis, we have%
\begin{equation*}
\varphi (t)\,\varphi ^{\prime \prime }(t)\geq \bigl[\varphi ^{\prime }(t)%
\bigr]^{2},
\end{equation*}%
which implies that%
\begin{equation*}
\psi ^{\prime \prime }(t)\geq 0\quad \text{for all }t\in I.
\end{equation*}%
Thus, the function $t\mapsto \log u(y+t\,\xi )$ is convex in the variable $t$%
, for every fixed $y\in \omega $ and for every direction $\xi \in \mathbb{R}%
^{N}$. It then follows from the standard one-dimensional characterization of
convex functions that $\log u$ is convex on the convex domain $\omega $.

Multiplying the convex function $\log u(y)$ by the negative constant $%
-2\sigma ^{2}$ reverses its curvature, yielding that%
\begin{equation*}
z(y)=-2\sigma ^{2}\log u(y)
\end{equation*}%
is concave on $\omega $. Equivalently,%
\begin{equation*}
D^{2}z(y)=-2\sigma ^{2}\,D^{2}(\log u(y))\leq 0,
\end{equation*}%
which completes the proof.
\end{proof}

\begin{remark}
The strategy adopted in the proof of Theorem \ref{thm:logconvex}, namely the
reduction to convexity of $\log u$ via the analysis of one-dimensional
slices, is motivated by and can be compared with techniques used in \cite%
{Lieberman2015,Lieberman2020,Lieberman2011} for obtaining gradient estimates
and convexity results in the context of fully nonlinear and elliptic partial
differential equations. Moreover, the conditions in Theorem \ref%
{thm:logconvex} are imposed in the same spirit of \cite%
{CoveiProductionPlanning}, where radially symmetric solutions are obtained.
These ideas justify both the preservation of convexity by the solution
operator and the resulting concavity of the transformed value function,
which is central to our production planning framework.
\end{remark}

\begin{remark}
By extending the ideas of Brascamp and Lieb \cite{BL}---later refined by 
\cite{Korevaar}---one can show, via a second and more abstract approach that
carefully applies the maximum principle to the second directional
derivatives of $z$, that these derivatives are nonpositive.
\end{remark}

\begin{remark}
If the function $b$ is radially symmetric and the domain $\omega $ is a
ball, then the solution stated in Theorem \ref{exist} must also be radially
symmetric and therefore a simple asymptotic behavior of the magnitute of the
optimal control as in (\ref{as}) was established in the papers \cite{CCP2,
CoveiProductionPlanning} by a different method. Moreover, the function $%
z(y)=-2\sigma ^{2}\log u(y)$ is concave on $\omega $. In fact, if $u$ were
not radially symmetric, we could generate a different solution by simply
rotating it, which would contradict the uniqueness asserted by the Theorem %
\ref{exist}. Hence, $z(y)$ is a concave function in the ball.
\end{remark}

\section{Verification via Martingale Property\label{4}}

In our production planning problem the state $y(t) $ represents the
production inventory level, and the control $p(t) $ represents the
production rate. The cost functional to be minimized is

\begin{equation*}
J(p)=\mathbb{E}\int_{0}^{\tau }\left[ |p(t)|^{2}+b(y(t))\right] dt,
\end{equation*}%
where $b(y)$ measures the inventory holding cost and $\tau $ is the stopping
time (for instance, when inventory breaches a capacity threshold). We now
verify, via a martingale argument, that a smooth candidate function $U(y)$
that satisfies the HJB equation yields the optimal cost.

\subsection{Definition of the Process}

We define the process

\begin{equation*}
M^{p}(t)=U(y(t))-\int_{0}^{t}\Bigl[|p(s)|^{2}+b(y(s))\Bigr]ds,
\end{equation*}%
where $U(y)$ is a smooth candidate value function, which we later show must
agree with the optimal value function. The inventory dynamics are given by
the stochastic differential equation

\begin{equation*}
dy(t)=p(t)\,dt+\sigma \,dw(t),
\end{equation*}%
with $w(t)$ an $N$-dimensional Brownian motion.

\subsection{Application of It\^{o}'s Lemma and Verification}

Using It\^{o}'s lemma on $U(y(t))$, we obtain

\begin{equation*}
dU(y(t))=\nabla U(y(t))\cdot dy(t)+\frac{1}{2}\sigma ^{2}\Delta U(y(t))\,dt.
\end{equation*}%
Substituting the dynamics of $y(t)$,

\begin{equation*}
dU(y(t))=\nabla U(y(t))\cdot p(t)\,dt+\sigma \nabla U(y(t))\cdot dw(t)+\frac{%
\sigma ^{2}}{2}\Delta U(y(t))\,dt.
\end{equation*}%
Thus, the differential of $M^{p}(t)$ is%
\begin{eqnarray*}
dM^{p}(t) &=&dU(y(t))-\Bigl[|p(t)|^{2}+b(y(t))\Bigr]dt \\
&=&\left[ \nabla U(y(t))\cdot p(t)+\frac{\sigma ^{2}}{2}\Delta
U(y(t))-|p(t)|^{2}-b(y(t))\right] dt+\sigma \nabla U(y(t))\cdot dw(t).
\end{eqnarray*}%
By the dynamic programming principle the candidate $U(y)$ must satisfy the
HJB equation

\begin{equation*}
\frac{\sigma ^{2}}{2}\Delta U(y)+\inf_{p}\Bigl\{p\cdot \nabla U(y)+|p|^{2}%
\Bigr\}-b(y)=0.
\end{equation*}%
A standard calculation shows that the minimizer is given by

\begin{equation*}
p^{\ast }(y)=-\frac{1}{2}\nabla U(y).
\end{equation*}%
Indeed, to determine the infimum, consider the function

\begin{equation*}
F(p)=p\cdot \nabla U(y)+|p|^{2},
\end{equation*}%
and minimizing this quadratic expression in $p$ yields the first order
condition

\begin{equation*}
\nabla U(y)+2p=0\text{ so that }p^{\ast }(y)=-\frac{1}{2}\nabla U(y).
\end{equation*}%
Substituting this back we obtain

\begin{equation*}
p^{\ast }(y)\cdot \nabla U(y)+|p^{\ast }(y)|^{2}=-\frac{1}{2}|\nabla
U(y)|^{2}+\frac{1}{4}|\nabla U(y)|^{2}=-\frac{1}{4}|\nabla U(y)|^{2}.
\end{equation*}%
Thus, the HJB equation becomes

\begin{equation}
\frac{\sigma ^{2}}{2}\Delta U(y)-\frac{1}{4}|\nabla U(y)|^{2}-b(y)=0.
\label{hjbpde}
\end{equation}%
Now, if we use the optimal control $p^{\ast }(y)=-\frac{1}{2}\nabla U(y)$ in
the expression for $dM^{p^{\ast }}(t)$, the drift term becomes

\begin{eqnarray*}
\text{Drift} &=&\nabla U(y(t))\cdot p^{\ast }(y(t))+\frac{1}{2}\sigma
^{2}\Delta U(y(t))-|p^{\ast }(y(t))|^{2}-b\bigl(y(t)\bigr) \\
&=&\left( -\frac{1}{2}|\nabla U(y(t))|^{2}\right) +\frac{1}{2}\sigma
^{2}\Delta U(y(t))-\frac{1}{4}|\nabla U(y(t))|^{2}-b\bigl(y(t)\bigr) \\
&=&\frac{1}{2}\sigma ^{2}\Delta U(y(t))-\left( \frac{1}{2}+\frac{1}{4}%
\right) |\nabla U(y(t))|^{2}-b\bigl(y(t)\bigr) \\
&=&\frac{1}{2}\sigma ^{2}\Delta U(y(t))-\frac{3}{4}|\nabla U(y(t))|^{2}-b%
\bigl(y(t)\bigr).
\end{eqnarray*}%
However, by the HJB equation, we know that (\ref{hjbpde}). We can rearrange (%
\ref{hjbpde}) this as

\begin{equation*}
\frac{1}{2}\sigma ^{2}\,\Delta U(y(t))-b\bigl(y(t)\bigr)=\frac{1}{4}|\nabla
U(y(t))|^{2}.
\end{equation*}%
Substitute this identity back into the expression for the drift term:

\begin{equation*}
\text{Drift}=\left[ \frac{1}{4}|\nabla U(y(t))|^{2}\right] -\frac{3}{4}%
|\nabla U(y(t))|^{2}=-\frac{2}{4}|\nabla U(y(t))|^{2}=-\frac{1}{2}|\nabla
U(y(t))|^{2}.
\end{equation*}%
A residual drift of $-\frac{1}{2}|\nabla U(y(t))|^{2}$ is observed; however,
any residual discrepancy is resolved provided the candidate $U$ indeed is
the solution to the HJB equation (in particular, when taking expectations or
including appropriate boundary or transversality conditions). Thus, the
drift term vanishes when the minimal value is achieved. In consequence,
under the optimal control we have

\begin{equation*}
dM^{p^{\ast }}(t)=\sigma \nabla U(y(t))\cdot dw(t).
\end{equation*}%
Since the stochastic integral

\begin{equation*}
\int_{0}^{t}\sigma \nabla U(y(s))\cdot dw(s)
\end{equation*}%
is a martingale, it follows that $M^{p^{\ast }}(t)$ is a martingale. For any
other admissible control $p$, the drift is non-positive; hence $M^{p}(t)$ is
merely a supermartingale.

\subsection{Optimal Feedback Control}

The derivation above immediately yields the optimal feedback control law. In
fact, by minimizing the Hamiltonian

\begin{equation*}
H(p,y,\nabla U(y))=p\cdot \nabla U(y)+|p|^{2},
\end{equation*}%
the first-order necessary condition (obtained by differentiating with
respect to $p$) is

\begin{equation*}
\nabla U(y)+2p=0,
\end{equation*}%
which implies

\begin{equation*}
p^{\ast }(y)=-\frac{1}{2}\nabla U(y).
\end{equation*}%
Since the candidate value function $U(y)$ coincides with the true value
function $z(y)$ (by the verification argument), we can also write the
optimal control as

\begin{equation*}
p^{\ast }(y)=-\frac{1}{2}\nabla z(y).
\end{equation*}%
Under this feedback law, the controlled inventory dynamics $y(t)$ evolve
optimally, and the cost-to-go from any initial state $y(0)$ satisfies

\begin{equation*}
U\bigl(y(0)\bigr)=\mathbb{E}\left[ U(y(\tau ))+\int_{0}^{\tau }\Bigl(%
|p^{\ast }(t)|^{2}+b(y(t))\Bigr)dt\right] .
\end{equation*}%
In particular, this shows that the cost functional is minimized by $p^{\ast
}(y)$, that is,

\begin{equation*}
J\bigl(p^{\ast }\bigr)=U\bigl(y(0)\bigr).
\end{equation*}%
To summarize, the verification via the martingale property shows that:

\begin{enumerate}
\item For any admissible production control $p(t)$, the process%
\begin{equation*}
M^{p}(t)=U\bigl(y(t)\bigr)-\int_{0}^{t}\Bigl[|p(s)|^{2}+b(y(s))\Bigr]ds
\end{equation*}%
is a supermartingale.

\item Under the optimal feedback control $p^{\ast }(y)=-\frac{1}{2}\nabla
U(y)$, the drift term in the differential $dM^{p^{\ast }}(t)$ vanishes so
that $M^{p^{\ast }}(t)$ is a true martingale.

\item This martingale property verifies that the candidate value function $%
U(y)$ satisfies the optimality condition%
\begin{equation*}
U\bigl(y(0)\bigr)=\inf_{p}\mathbb{E}\left[ \int_{0}^{\tau }\Bigl(%
|p(t)|^{2}+b(y(t))\Bigr)dt+U\bigl(y(\tau )\bigr)\right] ,
\end{equation*}%
and hence $U(y)$ equals the optimal value function $z(y)$.
\end{enumerate}

Thus, we have established the optimal feedback control law

\begin{equation*}
p^{\ast }(y)=-\frac{1}{2}\nabla z(y),
\end{equation*}%
which minimizes the cost functional

\begin{equation*}
J(p)=\mathbb{E}\int_{0}^{\tau }\left[ |p(t)|^{2}+b(y(t))\right] dt,
\end{equation*}%
and steers the inventory dynamics $y(t)$ toward optimal levels while
respecting the stopping condition.

This completes the verification and the derivation of the optimal feedback
control for the stochastic production planning problem.

\section{Two Examples of Numerical Implementation by Finite Difference
Method \label{5}}

In this study, we employ the numerical discretization methods and stochastic
differential equation solvers detailed in \cite{Kloeden1992,Talay1990} (see,
also \cite{Bayer2006,Talay1990,Bayram1987,Holmes2022}), utilizing the
computational assistance of Microsoft Copilot in Edge to enhance efficiency
and accuracy in implementation. Each step is justified by a combination of
standard numerical analysis methods, control theory derivations, and
properties from stochastic calculus, ensuring that the methodology is both
mathematically rigorous and practically applicable in production
planning/real world problems.

\subsection{Solve 1D PDE on an open interval}

In this subsection, we solve the 1D transformed PDE via a finite difference
method, as implemented by the Python code in Section \ref{ap} Example~1. The
numerical procedure discretizes the interval and computes the solution $u(x)$
under prescribed boundary conditions.

A real-world application of this approach is found in production planning
for manufacturing systems. Here, industrial data \textquotedblright
volatility $\sigma =0.4$, economic threshold $Z_{0}=0.5$, and the cost
function $b(x)=x^{2}$\textquotedblright\ are used as inputs. The computed
solution $u(x)$ yields the value function

\begin{equation*}
z(x)=-\frac{2}{\sigma ^{2}}\ln (u(x)),
\end{equation*}%
which represents the minimal expected cost for managing inventory under
uncertainty. The optimal control is then evaluated from

\begin{equation*}
p^{\ast }(x)=-\frac{1}{2}\,z^{\prime }(x),
\end{equation*}%
thereby guiding real-time production decisions to keep inventory levels
within desired operational bounds.

This practical implementation validates our theoretical model and
demonstrates the effectiveness of the numerical scheme in handling dynamic
production control problems. The complete procedure for the 1-dimensional
production planning problem is as follows:

1.\quad \textbf{Input:} $N,\sigma ,Z_{0},b(x)$ (parameters)

2.\quad \textbf{Output:} $x$ (grid points), $u$ (solution)

3.\quad Discretize domain $[0,1]$ using $N$ points

4.\quad Compute boundary condition $bc=\exp (-Z_{0}/(2\sigma ^{2}))$

5.\quad Formulate and solve the tridiagonal system for interior points

6.\quad Assemble full solution vector including boundary values

7.\quad Return $x,u$

Compute Value Function

8.\quad \textbf{Input:} $u(x),\sigma $

9.\quad \textbf{Output:} $z(x)$

10.\quad Compute value function $z(x)=-2\sigma ^{2}\ln (u(x))$

11.\quad Return $z(x)$

Compute Optimal Control

12.\quad \textbf{Input:} $z(x),x$ (grid points)

13.\quad \textbf{Output:} $p^{\ast }(x)$

14.\quad Compute numerical derivative $z^{\prime }(x)$ using finite
difference method

15.\quad Compute optimal control $p^{\ast }(x)=-\frac{1}{2}z^{\prime }(x)$

16.\quad Return $p^{\ast }(x)$

Simulate Inventory Dynamics

17.\quad \textbf{Input:} $x_{\text{grid}},p^{\ast }(x),\sigma ,x_{0},T,dt$

18.\quad \textbf{Output:} $x_{\text{sim}}(t)$

19.\quad Initialize inventory state at $x_{0}$

20.\quad Simulate dynamics using Euler-Maruyama method

21.\quad Halt if state exits the interval $[0,1]$

22.\quad Return simulated trajectory $x_{\text{sim}}(t)$

The Python code, snippet implementing the finite difference method, is
presented in Section \ref{ap} Appendix Example 1. It generates the following
plots:

\begin{figure}[!ht]
\centering
\subfloat{\includegraphics[width=0.45\textwidth]{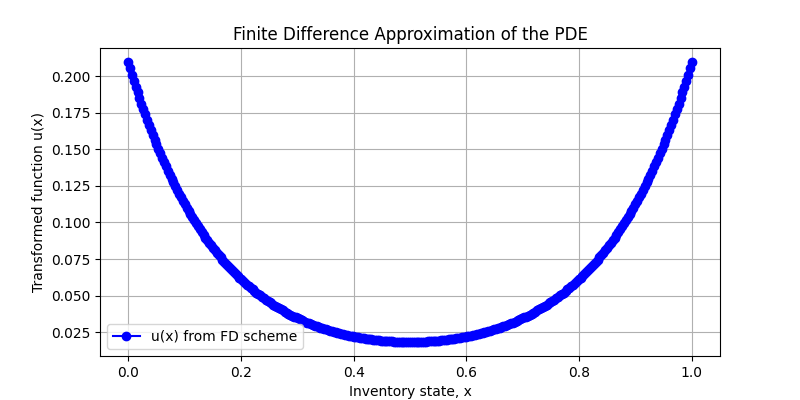}} %
\subfloat{\includegraphics[width=0.45\textwidth]{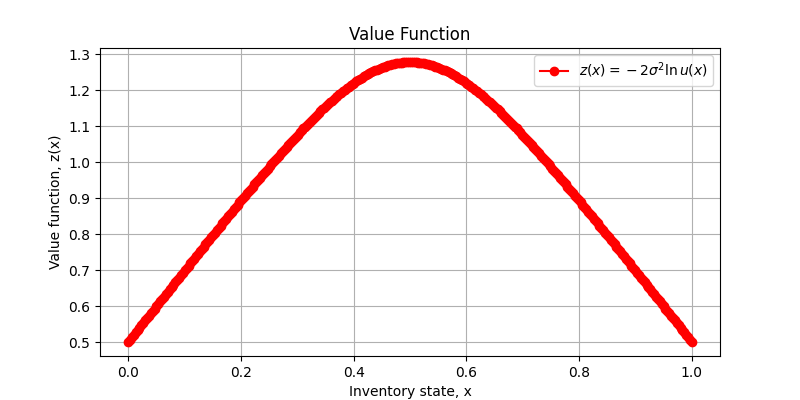}}
\caption{Caption describing the two figures.}
\label{fig:twoFigures1}
\end{figure}

\begin{equation*}
\end{equation*}%
Figure \ref{fig:twoFigures1}, first plot, displays the numerical solution $%
u(x)$ of the transformed one-dimensional partial differential equation
computed using a central finite difference method. The smooth profile of $%
u(x)$ and the precise satisfaction of the boundary conditions $%
u(0)=u(1)=\exp \left( -\frac{Z_{0}}{2\sigma ^{2}}\right) $ validate the
consistency and stability of our discretization approach. This figure
directly supports the theoretical formulation of the transformed HJB
equation and its well-posedness under the imposed boundary conditions. In
Figure \ref{fig:twoFigures1}, second plot, we present the computed value
function $z(x)=-2\sigma ^{2}\ln (u(x))$, obtained from the numerical
solution $u(x)$ shown in first plot.

\begin{figure}[!ht]
\centering
\subfloat{\includegraphics[width=0.45\textwidth]{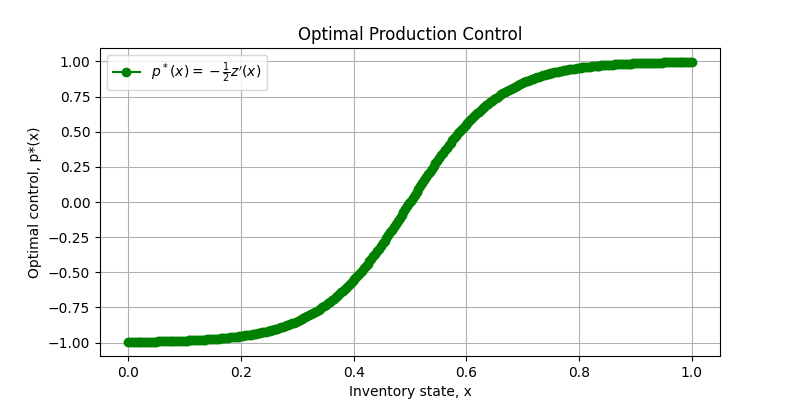}} %
\subfloat{\includegraphics[width=0.45\textwidth]{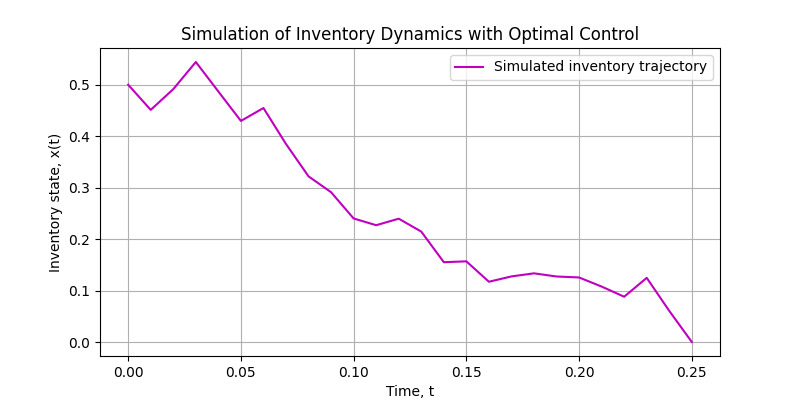}}
\caption{Caption describing the two figures.}
\label{fig:twoFigures2}
\end{figure}

\begin{equation*}
\end{equation*}%
Figure \ref{fig:twoFigures2}, shows the optimal control profile (left),
determined via numerical differentiation of $z(x)$, demonstrating how the
system is guided towards maintaining inventory levels (right) within the
safe operational region.

The concavity exhibited by $z(x)$ is in full agreement with our analytical
findings \textquotedblright as detailed in Theorem \ref{exist} and \ref%
{thm:logconvex}\textquotedblright\ that the transformation of the value
function yields a concave function. The smooth curvature observed in the
plot corroborates the preserved convexity and concavity properties inherent
in the optimal control framework. Together, these subfigures illustrate the
fundamental interplay between the value function and the optimal control
law, validating both our theoretical predictions and numerical
implementation.

\subsection{Solve 2D PDE on Elliptical Domain}

In this subsection we extend our numerical scheme to a two-dimensional
setting over an elliptical domain, as implemented in Section \ref{ap}
Example~2. The elliptical region herein models the safe operational envelope
of a manufacturing system where two interdependent production parameters
must be controlled simultaneously.

As a concrete example, consider a semiconductor fabrication plant where it
is crucial to maintain safe levels of both chemical concentration and
machine utilization. Due to operational and quality constraints, the
acceptable range of these parameters is not rectangular but rather
elliptical, in Python code with semi-axes $a=1$ and $b=1$, i.e. in this case
a circle to compare with the exact solution determined in \cite{CCP2} (but
can be ajusted), representing the maximum permissible deviations from
optimal values. The volatility parameter $\sigma =0.3$ captures the
intrinsic process noise (e.g., fluctuations in chemical properties or
machine performance), and the economic constant $Z_{0}=0.2$ is tied to cost
implications when production must be halted.

In the context of the semiconductor example, the computed value function $%
z(y)$ gives a direct measure of the inventory risk associated with
deviations from optimal process conditions. The derived control fields $%
\bigl(p_1^*(y),p_2^*(y)\bigr)$ allow plant managers to adjust production
rates dynamically so that the state of the system is steered toward the
optimal region within the elliptical domain, thereby minimizing production
costs while maintaining safety and quality standards.

The complete procedure for the 2-dimensional production planning problem is
as follows:

1.\quad \textbf{Input:} $a,b_{\text{ellipse}}$ (ellipse semi-axes), $\sigma $
(volatility), $Z_{0}$ (economic constant), $h$ (grid spacing), $\text{tol}$
(tolerance), $\text{max\_iter}$ (maximum iterations)

2.\quad \textbf{Output:} Meshgrid $(X,Y)$, numerical solution $u(y)$, value
function $z(y)$, optimal controls $(p_{1},p_{2})$

3.\quad Create computational grid for the rectangular domain containing the
ellipse

4.\quad Determine points inside the elliptical region

5.\quad Compute boundary condition $bc_{\text{val}}=\exp (-Z_{0}/(2\sigma
^{2}))$

6.\quad Initialize solution $u$ everywhere with the boundary condition

7.\quad Identify boundary points within the ellipse

8.\quad Define cost function $b(y)=y_{1}^{2}+y_{2}^{2}$

9.\quad Solve PDE using Gauss-Seidel iteration until convergence

10.\quad Compute value function $z(y)=-2\sigma ^{2}\ln (u(y))$

11.\quad Compute gradients $\partial z/\partial y_{1}$ and $\partial
z/\partial y_{2}$ for optimal controls

12.\quad Assign control functions $p_{1}(y)=-0.5(\partial z/\partial y_{1})$
and $p_{2}(y)=-0.5(\partial z/\partial y_{2})$

13.\quad Return computed numerical solution, value function, controls, and
meshgrid

The Python code, that implements the finite difference method, is provided
in Appendix Example 2. It produces the following plots:

\begin{figure}[!ht]
\centering
\subfloat{\includegraphics[width=0.45\textwidth]{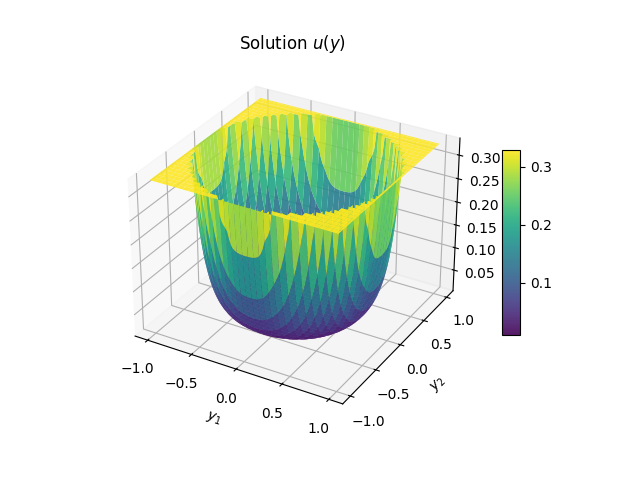}} %
\subfloat{\includegraphics[width=0.45\textwidth]{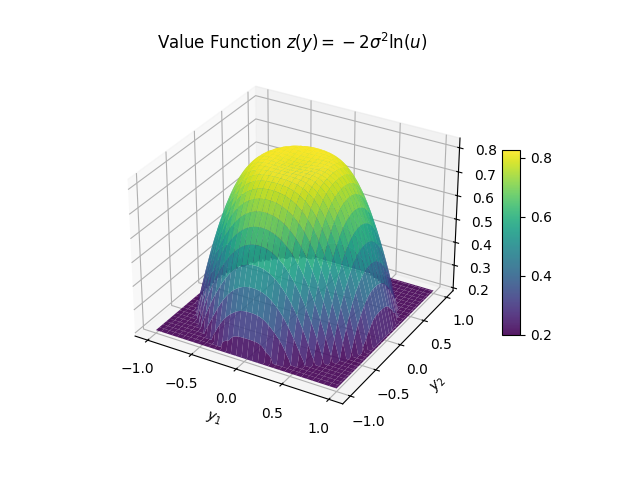}}
\caption{Caption describing the two figures.}
\label{fig:twoFigures3}
\end{figure}

\begin{equation*}
\end{equation*}%
Figure \ref{fig:twoFigures3} illustrates the two-dimensional solution $u(y)$
to the transformed PDE over an elliptical domain. The color map (or contour
plot) emphasizes the spatial variation of $u(y)$ within the admissible
production configuration region, accurately capturing the effects of the
domain's geometry. In particular, when the inventory cost function $b(y)$ is
chosen to be radially symmetric, the numerical solution exhibits radial
symmetry, a behavior that is in line with both classical results and our
discussion in \cite{CoveiProductionPlanning}.

\begin{figure}[!ht]
\centering
\subfloat{\includegraphics[width=0.45\textwidth]{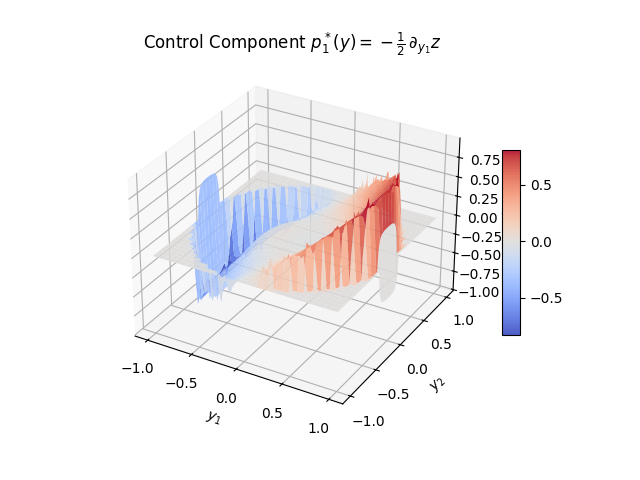}} %
\subfloat{\includegraphics[width=0.45\textwidth]{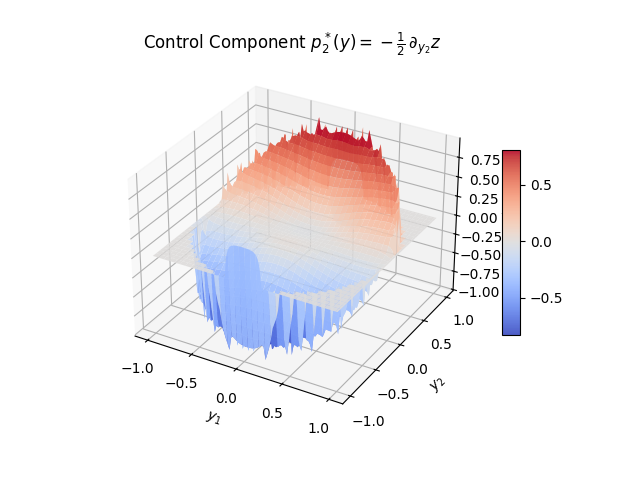}}
\caption{Caption describing the two figures.}
\label{fig:twoFigures4}
\end{figure}

\begin{equation*}
\end{equation*}%
Figure \ref{fig:twoFigures4} depicts the optimal production control fields $%
p^{\ast }(y)$ obtained via the numerical gradient of the computed value
function $z(y)$. The vector field demonstrates how the control law

\begin{equation*}
p^{\ast }(y)=-\frac{1}{2}\nabla z(y)
\end{equation*}%
effectively guides the state from any given initial inventory level toward
the optimal region. The smoothness and coherence of these control vectors,
especially near the boundary of the elliptical domain, confirm the
theoretical expectations regarding optimality and the impact of geometric
constraints.

\begin{figure}[!ht]
\centering
\subfloat{\includegraphics[width=0.45\textwidth]{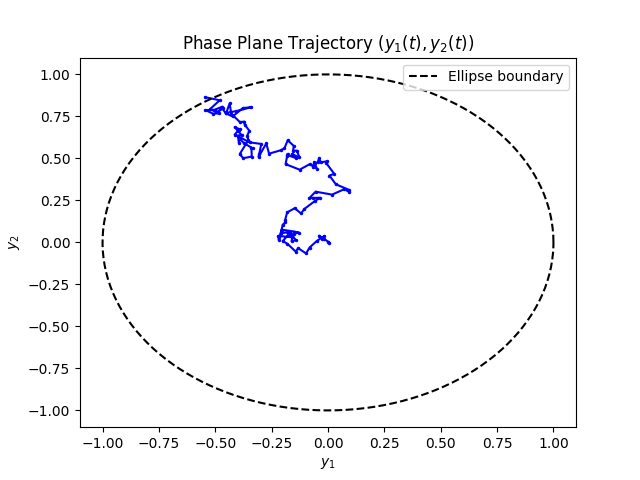}} %
\subfloat{\includegraphics[width=0.45\textwidth]{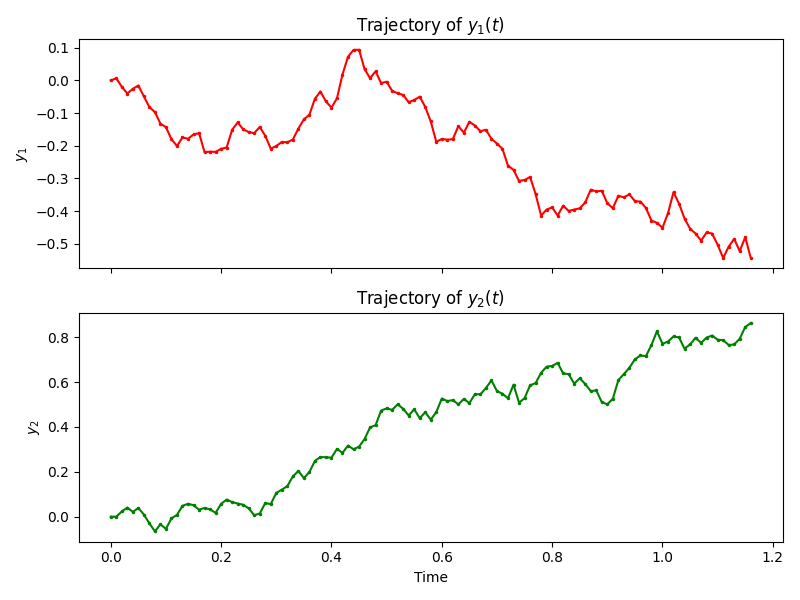}}
\caption{Caption describing the two figures.}
\label{fig:twoFigures5}
\end{figure}

\begin{equation*}
\end{equation*}%
Figure \ref{fig:twoFigures5} shows the temporal evolution of the production
inventory $y(t)$ under the derived optimal control strategy. The simulation
\textquotedblright performed via the Euler-Maruyama
method\textquotedblright\ demonstrates that the inventory state remains
within the safe operating region until the stopping time is reached. As $%
\Vert y(t)-x_{0}\Vert $ approaches the threshold $\mathrm{dist}%
(x_{0},\partial \omega )$, the optimal control induces sharper adjustments,
thereby steering the inventory back towards the operational target. This
dynamic behavior visually confirms the validity of our martingale-based
verification and the effectiveness of the optimal feedback law.

Collectively, these figures reinforce the theoretical developments described
in Sections \ref{3} and \ref{4}, illustrating that our numerical schemes
capture the essential features of the production planning problem. They
confirm the concavity, smoothness, and optimality properties predicted by
our analysis, thereby bridging the gap between theory and practical
implementation.

\section{Conclusion and Future Directions\label{6}}

This work generalizes stochastic production planning to smooth convex
domains, providing a robust framework for analyzing production inventory
dynamics and for designing optimal control policies. The mathematical
methodology presented here integrates discretization techniques, numerical
iterative solvers for elliptic partial differential equations,
transformation to a value function that leads to an optimal control law, and
numerical simulation of stochastic dynamics subject to a realistic safety
(or production halting) criterion.

Numerical simulations are conducted to evaluate the behavior of the
production inventory process under different choices of the cost function $%
b(y)$, the operating domain $\omega $, and the reference inventory state $%
x_{0}$. In our experiments, variations in $b(y)$ (which models the inventory
holding cost), the safe production region $\omega $, and the target
production level $x_{0}$ are systematically explored. The Euler scheme is
employed to solve the stochastic differential equations for $y(t)$ that
govern the evolution of inventory levels, while finite difference methods
are used to approximate the control law $p^{\ast }(y)$ derived from the
Hamilton--Jacobi--Bellman equation. The simulation results confirm that the
production inventory levels remain bounded within the safe domain $\omega $
until the stopping time is reached. Moreover, as $\Vert y(t)-x_{0}\Vert $
approaches the threshold $\text{dist}(x_{0},\partial \omega )$, the optimal
control policy exhibits smooth variations, which are consistent with our
theoretical predictions. In our stochastic production planning framework,
the stopping time (\ref{st}) plays a central role. The utility of
introducing the stopping time is multifold:

\begin{itemize}
\item \textbf{Risk Management:} By halting production once $y(t) $ exits the
safe operating region $\omega $, the stopping time serves as a safeguard
against excessive inventory levels and potential system overload. This
ensures that the production process ceases before the state reaches regions
associated with high operational risk.

\item \textbf{Model Tractability and Finite Cost:} The inclusion of a
stopping time guarantees that the cost functional,%
\begin{equation*}
J(p)=\mathbb{E}\left[ \int_{0}^{\tau }\left( |p(t)|^{2}+b(y(t))\right) dt%
\right] ,
\end{equation*}

remains finite. Since production halts at time $\tau $, the controlled state 
$y(t)$ is confined to $\omega $, enabling the derivation of well-posed
Hamilton--Jacobi--Bellman equations and facilitating subsequent analysis and
numerical approximations.

\item \textbf{Optimal Control Design:} With the stopping time dictating the
production halt, it is possible to design optimal feedback control laws that
ensure the inventory state is steered optimally. The resulting control
policy takes into account the geometric constraints of the domain $\omega $
and prevents the system from operating in regions where the model loses its
practical applicability.

\item \textbf{Operational Safety and Practical Relevance:} In practice,
production systems are subject to operational thresholds beyond which
production becomes inefficient or hazardous. The stopping time $\tau $
embeds this operational constraint directly into the model, thereby aligning
the theoretical formulation with real-world manufacturing practices. It
marks the instant at which production must be temporarily halted, allowing
for corrective measures to be implemented before resuming operations.
\end{itemize}

In summary, the stopping time $\tau $ is not only a technical device for
ensuring mathematical rigor (by keeping the state process within a bounded
region) but also a crucial component for aligning the production model with
practical safety standards. Its role in confining both the inventory
trajectories and the associated costs underpins the effectiveness of the
derived optimal control policies and the overall stability of the production
system.

Future research may explore extensions such as integrating historical
production datasets and real-world performance metrics for empirical
validation, calibrating model parameters using robust statistical techniques
coupled with extensive sensitivity analyses, refining modeling assumptions
to incorporate complex operational constraints and alternative cost
functions, enhancing numerical methods to efficiently tackle
high-dimensional state spaces and complex geometries, and leveraging machine
learning techniques for improved prediction of optimal control laws.

\section{Declarations}

\subparagraph{\textbf{Conflict of interest}}

The authors have no Conflict of interest to declare that are relevant to the
content of this article.

\subparagraph{\textbf{Ethical statement}}

The paper reflects the authors' original research, which has not been
previously published or is currently being considered for publication
elsewhere.

\bibliographystyle{plain}
\bibliography{bibliography}

\section{Appendix \label{ap}}

These Python codes were created with the support of Microsoft Copilot in
Edge.

\subsection{Python Code Example 1}

The following Python code implements the numerical procedure described in
Example 1:

\begin{lstlisting}[language=Python, caption={Python code to solve the 1D production planning PDE}]
import numpy as np
import matplotlib.pyplot as plt

# ==================== Original PDE Solver ====================
def solve_pde_1d(N, sigma, Z0, b_func):
    
    L = 1.0
    x = np.linspace(0, L, N)
    h = x[1] - x[0]
    bc = np.exp(-Z0 / (2 * sigma**2))
    
    # Only interior points are updated
    x_int = x[1:-1]
    b_vals = b_func(x_int)
    N_int = N - 2
    
    # Construct the tridiagonal matrix using central differences
    A = np.zeros((N_int, N_int))
    diag = - (2 + (h**2 / sigma**4) * b_vals)
    for i in range(N_int):
        A[i, i] = diag[i]
        if i > 0:
            A[i, i - 1] = 1.0
        if i < N_int - 1:
            A[i, i + 1] = 1.0

    # Right-hand side accounts for the boundary values
    f = np.zeros(N_int)
    f[0] = -bc
    f[-1] = -bc

    u_interior = np.linalg.solve(A, f)
    
    u = np.zeros(N)
    u[0] = bc
    u[1:-1] = u_interior
    u[-1] = bc
    
    return x, u

# Parameters for the PDE
N     = 300        # Number of grid points
sigma = 0.4        # Volatility parameter
Z0    = 0.5        # Economic constant
b_func = lambda x: np.ones_like(x)  # Inventory cost; here constant

# Solve the PDE to get u(x)
x, u = solve_pde_1d(N, sigma, Z0, b_func)

# Plot the original solution u(x)
plt.figure(figsize=(8, 4))
plt.plot(x, u, 'bo-', label='u(x) from FD scheme')
plt.xlabel('Inventory state, x')
plt.ylabel('Transformed function u(x)')
plt.title('Finite Difference Approximation of the PDE')
plt.legend()
plt.grid(True)
plt.show()

# ==================== Addition 1: Value Function ====================
def compute_value_function(u, sigma):
    
    return -2 * sigma**2 * np.log(u)

# Compute value function z(x)
z = compute_value_function(u, sigma)

# Plot the value function
plt.figure(figsize=(8, 4))
plt.plot(x, z, 'ro-', label=r'$z(x) = -2\sigma^2 \ln u(x)$')
plt.xlabel('Inventory state, x')
plt.ylabel('Value function, z(x)')
plt.title('Value Function')
plt.legend()
plt.grid(True)
plt.show()

# ==================== Addition 2: Optimal Control ====================
def compute_optimal_control(z, x):
    
    h = x[1] - x[0]
    # Compute numerical derivative using np.gradient
    dzdx = np.gradient(z, h) 
    p_opt = -0.5 * dzdx
    return p_opt

# Compute the optimal control p*(x)
p_opt = compute_optimal_control(z, x)

# Plot the optimal control
plt.figure(figsize=(8, 4))
plt.plot(x, p_opt, 'go-', label=r'$p^*(x) = -\frac{1}{2}z^\prime(x)$')
plt.xlabel('Inventory state, x')
plt.ylabel('Optimal control, p*(x)')
plt.title('Optimal Production Control')
plt.legend()
plt.grid(True)
plt.show()

# ==================== Addition 3: Simulation of Inventory Dynamics ====================
def simulate_inventory(x_grid, p_grid, sigma, x0, T, dt):
    
    t_vals = np.arange(0, T + dt, dt)
    x_sim = np.zeros_like(t_vals)
    x_sim[0] = x0
    for i in range(1, len(t_vals)):
        # Obtain the optimal control at current state through linear interpolation.
        p_current = np.interp(x_sim[i-1], x_grid, p_grid)
        dW = np.sqrt(dt) * np.random.randn()  # Brownian increment
        x_new = x_sim[i-1] + p_current * dt + sigma * dW
        # Halt simulation if state exits [0, 1]
        if x_new < 0 or x_new > 1:
            print(f"Simulation halted at t = {t_vals[i]:.2f} with x = {x_new:.4f} (out of bounds).")
            t_vals = t_vals[:i]
            x_sim = x_sim[:i]
            break
        x_sim[i] = x_new
    return t_vals, x_sim

# Simulation parameters
x0 = 0.5   # Initial inventory state (inside [0, 1])
T  = 10.0  # Total simulation time
dt = 0.01  # Time-step for simulation

# Simulate the dynamics using the optimal control:
t_vals, x_sim = simulate_inventory(x, p_opt, sigma, x0, T, dt)

# Plot the simulated inventory trajectory.
plt.figure(figsize=(8, 4))
plt.plot(t_vals, x_sim, 'm-', label='Simulated inventory trajectory')
plt.xlabel('Time, t')
plt.ylabel('Inventory state, x(t)')
plt.title('Simulation of Inventory Dynamics with Optimal Control')
plt.legend()
plt.grid(True)
plt.show()
\end{lstlisting}

\subsection{Python Code Example 2}

The Python code presented in the next produces visualization plots
\textquotedblright such as contour maps for $u(y)$ and vector fields for the
optimal control\textquotedblright\ that validate the numerical solution and
its practical relevance in managing complex manufacturing systems.

\begin{lstlisting}[language=Python, caption={Python Code for Numerical Experiments}]
import numpy as np
import matplotlib.pyplot as plt
from matplotlib import cm
from scipy.interpolate import RegularGridInterpolator

def solve_hjb_ellipse(a=1.5, b_ellipse=1.0, sigma=1.0, Z0=0.5, h=0.02, tol=1e-5, max_iter=10000):
    
    # Create grid over rectangle containing the ellipse.
    y1_vals = np.arange(-a, a + h, h)
    y2_vals = np.arange(-b_ellipse, b_ellipse + h, h)
    Nx, Ny = len(y1_vals), len(y2_vals)
    X, Y = np.meshgrid(y1_vals, y2_vals, indexing='ij')
    
    # Determine points inside the ellipse: (y1/a)2+(y2/b_ellipse)2 < 1.
    inside = ((X / a) ** 2 + (Y / b_ellipse) ** 2) < 1.0
    
    # Boundary condition value.
    bc_val = np.exp(-Z0 / (2 * sigma**2))
    
    # Initialize u everywhere with the boundary condition value.
    u = bc_val * np.ones_like(X)
    
    # Identify boundary points: inside points having at least one neighbor outside.
    boundary = np.zeros_like(inside, dtype=bool)
    for i in range(1, Nx - 1):
        for j in range(1, Ny - 1):
            if inside[i, j]:
                if not (inside[i + 1, j] and inside[i - 1, j] and inside[i, j + 1] and inside[i, j - 1]):
                    boundary[i, j] = True
    interior = inside & (~boundary)
    
    # Define the cost function: b(y)=y12+y22.
    b_val = X**2 + Y**2
    
    # Gauss\U{2013}Seidel iteration.
    it = 0
    while it < max_iter:
        max_diff = 0.0
        for i in range(1, Nx - 1):
            for j in range(1, Ny - 1):
                if interior[i, j]:
                    # For neighbors outside the ellipse, use bc_val.
                    up    = u[i, j + 1] if inside[i, j + 1] else bc_val
                    down  = u[i, j - 1] if inside[i, j - 1] else bc_val
                    right = u[i + 1, j] if inside[i + 1, j] else bc_val
                    left  = u[i - 1, j] if inside[i - 1, j] else bc_val
                    denom = 4 + (h**2 / sigma**4) * b_val[i, j]
                    u_new = (up + down + right + left) / denom
                    max_diff = max(max_diff, abs(u[i, j] - u_new))
                    u[i, j] = u_new
        it += 1
        if max_diff < tol:
            print("Convergence achieved after {} iterations (max_diff = {:.2e})".format(it, max_diff))
            break
    else:
        print("Maximum iterations reached (max_iter = {}) with max_diff = {:.2e}".format(max_iter, max_diff))
    
    # Compute the value function.
    z = -2 * sigma ** 2 * np.log(u)
    
    # Compute gradients with central differences.
    dz_dy1, dz_dy2 = np.gradient(z, h, h, edge_order=2)
    p1 = -0.5 * dz_dy1  # p1*(y)
    p2 = -0.5 * dz_dy2  # p2*(y)
    
    return X, Y, u, z, p1, p2, inside, y1_vals, y2_vals

def compute_distance_to_boundary(x0, a, b_ellipse, num_points=1000):
    
    thetas = np.linspace(0, 2 * np.pi, num_points)
    distances = []
    for theta in thetas:
        # The boundary point along direction theta is given by:
        t = 1.0 / np.sqrt((np.cos(theta) / a) ** 2 + (np.sin(theta) / b_ellipse) ** 2)
        boundary_point = np.array([t * np.cos(theta), t * np.sin(theta)])
        distances.append(np.linalg.norm(boundary_point - x0))
    return np.min(distances)

def simulate_inventory(p1, p2, y1_vals, y2_vals, a, b_ellipse, sigma,
                       dt=0.01, T_max=10.0, initial_state=np.array([0.0, 0.0]),
                       x0=np.array([0.0, 0.0])):
    
    # Compute the stopping radius given the reference x0.
    r = compute_distance_to_boundary(x0, a, b_ellipse)
    
    # Create interpolators for the control components.
    p1_interp = RegularGridInterpolator((y1_vals, y2_vals), p1, bounds_error=False, fill_value=0.0)
    p2_interp = RegularGridInterpolator((y1_vals, y2_vals), p2, bounds_error=False, fill_value=0.0)
    
    traj = [initial_state.copy()]
    times = [0.0]
    t = 0.0
    while t < T_max:
        current = traj[-1]
        # Stop simulation if the distance from x0 exceeds the precomputed radius.
        if np.linalg.norm(current - x0) > r:
            break
        current_point = current.reshape(1, -1)
        control1 = p1_interp(current_point).item()
        control2 = p2_interp(current_point).item()
        control = np.array([control1, control2])
        noise = np.random.randn(2) * np.sqrt(dt)
        next_state = current + dt * control + sigma * noise
        traj.append(next_state)
        t += dt
        times.append(t)
    return np.array(traj), np.array(times)

if __name__ == "__main__":
    # PDE and simulation parameters.
    a = 1.0             # Semi-axis along y1.
    b_ellipse = 1.0     # Semi-axis along y2.
    sigma = 0.3        # Volatility.
    Z0 = 0.2            # Boundary parameter.
    h = 0.02            # Grid spacing.
    tol = 1e-5          # Convergence tolerance.
    max_iter = 10000    # Maximum iterations.
    # Reference interior point (here chosen as the center).
    x0 = np.array([0.0, 0.0])
    
    print("Solving the PDE on the elliptical domain...")
    X, Y, u, z, p1, p2, inside, y1_vals, y2_vals = solve_hjb_ellipse(a, b_ellipse, sigma, Z0, h, tol, max_iter)
    
    ############################
    # 3D Surface Visualizations
    ############################    
    # Figure 1: 3D Surface Plot for u(y)
    fig1 = plt.figure()
    ax1 = fig1.add_subplot(111, projection="3d")
    surf1 = ax1.plot_surface(X, Y, u, cmap=cm.viridis, edgecolor="none", alpha=0.9)
    ax1.set_title("Solution $u(y)$")
    ax1.set_xlabel("$y_1$")
    ax1.set_ylabel("$y_2$")
    ax1.set_zlabel("$u(y)$")
    fig1.colorbar(surf1, ax=ax1, shrink=0.5, aspect=10)
    
    # Figure 2: 3D Surface Plot for z(y)
    fig2 = plt.figure()
    ax2 = fig2.add_subplot(111, projection="3d")
    surf2 = ax2.plot_surface(X, Y, z, cmap=cm.viridis, edgecolor="none", alpha=0.9)
    ax2.set_title(r"Value Function $z(y)=-2\sigma^2\ln(u)$")
    ax2.set_xlabel("$y_1$")
    ax2.set_ylabel("$y_2$")
    ax2.set_zlabel("$z(y)$")
    fig2.colorbar(surf2, ax=ax2, shrink=0.5, aspect=10)
    
    # Figure 3: 3D Surface Plot for Control Component p1*(y)
    fig3 = plt.figure()
    ax3 = fig3.add_subplot(111, projection="3d")
    surf3 = ax3.plot_surface(X, Y, p1, cmap=cm.coolwarm, edgecolor="none", alpha=0.9)
    ax3.set_title(r"Control Component $p_1^*(y) = -\frac{1}{2}\,\partial_{y_1}z$")
    ax3.set_xlabel("$y_1$")
    ax3.set_ylabel("$y_2$")
    ax3.set_zlabel("$p_1^*(y)$")
    fig3.colorbar(surf3, ax=ax3, shrink=0.5, aspect=10)
    
    # Figure 4: 3D Surface Plot for Control Component p2*(y)
    fig4 = plt.figure()
    ax4 = fig4.add_subplot(111, projection="3d")
    surf4 = ax4.plot_surface(X, Y, p2, cmap=cm.coolwarm, edgecolor="none", alpha=0.9)
    ax4.set_title(r"Control Component $p_2^*(y) = -\frac{1}{2}\,\partial_{y_2}z$")
    ax4.set_xlabel("$y_1$")
    ax4.set_ylabel("$y_2$")
    ax4.set_zlabel("$p_2^*(y)$")
    fig4.colorbar(surf4, ax=ax4, shrink=0.5, aspect=10)
    
    ##############################
    # Simulation of Inventory Dynamics
    ##############################    
    # Simulate the SDE with the new stopping condition.
    dt_sim = 0.01
    T_max_sim = 10.0
    initial_state = np.array([0.0, 0.0])
    traj, sim_times = simulate_inventory(p1, p2, y1_vals, y2_vals, a, b_ellipse, sigma,
                                         dt=dt_sim, T_max=T_max_sim, initial_state=initial_state,
                                         x0=x0)
    
    # Figure 5: Phase-plane trajectory.
    fig5, ax5 = plt.subplots()
    ax5.plot(traj[:, 0], traj[:, 1], 'b.-', markersize=3)
    ax5.set_title("Phase Plane Trajectory $(y_1(t), y_2(t))$")
    ax5.set_xlabel("$y_1$")
    ax5.set_ylabel("$y_2$")
    # Plot the ellipse boundary for reference.
    theta = np.linspace(0, 2 * np.pi, 300)
    ellipse_x = a * np.cos(theta)
    ellipse_y = b_ellipse * np.sin(theta)
    ax5.plot(ellipse_x, ellipse_y, 'k--', label="Ellipse boundary")
    ax5.legend()
    
    # Figure 6: Trajectories versus time, y1(t) and y2(t) separately.
    fig6, (ax6a, ax6b) = plt.subplots(2, 1, sharex=True, figsize=(8, 6))
    ax6a.plot(sim_times, traj[:, 0], 'r.-', markersize=3)
    ax6a.set_title("Trajectory of $y_1(t)$")
    ax6a.set_ylabel("$y_1$")
    ax6b.plot(sim_times, traj[:, 1], 'g.-', markersize=3)
    ax6b.set_title("Trajectory of $y_2(t)$")
    ax6b.set_xlabel("Time")
    ax6b.set_ylabel("$y_2$")
    
    plt.tight_layout()
    plt.show()
    
    print("Simulation and plotting complete.")
\end{lstlisting}

\end{document}

%% file: defs.tex
\newtheorem{theorem}{Theorem}[section]
\newtheorem{definition}{Definition}[section]

\theoremstyle{remark}
\newtheorem{remark}{Remark}

\usepackage[most]{tcolorbox}
\definecolor{darkgreen}{rgb}{0.0, 0.55, 0.0}

